\documentclass[a4paper,oneside,11pt,reqno]{amsart}
\usepackage{amsthm}
\usepackage{amsfonts}
\usepackage{amsmath}
\usepackage{amssymb}
\usepackage{mathrsfs}
\usepackage{epsfig}
\usepackage{pst-plot}
\usepackage{ifthen}
\usepackage[active]{srcltx}
\usepackage{tikz}
\allowdisplaybreaks

\newcommand{\beq}{\begin{eqnarray}}
\newcommand{\eeq}{\end{eqnarray}}
\newcommand{\beqn}{\begin{eqnarray*}}
\newcommand{\eeqn}{\end{eqnarray*}}
\newcommand{\rar}{\rightarrow}

\reversemarginpar  

\numberwithin{equation}{section}

\newtheorem{theorem}{\bf Theorem}[section]
\newtheorem{proposition}[theorem]{\bf Proposition}
\newtheorem{corollary}[theorem]{\bf Corollary}

\theoremstyle{remark}
\newtheorem{definition}[theorem]{\bf Definition}
\newtheorem{example}[theorem]{\bf Example}
\newtheorem{remark}[theorem]{\bf Remark}

\newcommand*{\Ge}{\geqslant}

\newcommand*{\inp}[2]{\langle{#1},\,{#2} \rangle}
\newcommand*{\Le}{\leqslant}

\begin{document}

\title[Helson matrices induced by measures]
{Helson matrices induced by measures}

   \author[S. Chavan]{Sameer Chavan}
   \address{Department of Mathematics and Statistics\\
Indian Institute of Technology Kanpur, India}
   \email{chavan@iitk.ac.in}
   \author[C. K. Sahu]{Chaman Kumar Sahu}
   \address{Department of Mathematics, Indian Institute of Technology Bombay, Powai,
   	Mumbai, 400076, India}
\email{chamanks@math.iitb.ac.in, sahuchaman9@gmail.com}
 \author[K. B. Sinha]{Kalyan
B. Sinha
}
   \address{J. N. Centre for Advanced Scientific Research, Jakkur, Bangalore 560064, India  \\ and
 Indian Statistical Institute, India
  }
\email{kbs@jncasr.ac.in}


\begin{abstract} 
We discuss the boundedness, Schatten-class properties and scattering theory of Helson matrices. We also discuss a class of Helson matrices induced by positive and signed measures.
All the results of this paper are illustrated with several examples not considered earlier.
\end{abstract}

\maketitle

\section{Introduction}

This paper is motivated by the investigations in \cite{BPSSV, MP, PP-1, PP} surrounding the notion of the multiplicative Hankel matrix, commonly known as the {\it Helson matrix}. The $(m, n)$-th entry of a Helson matrix depends only on the product $mn.$ The purpose of this work is to discuss basic theory of Helson matrices induced by the Laplace transform of a class of positive measures. It is worth noting that our methods also allow us to consider Helson matrices induced by signed measures.

The sets of positive integers and non-negative integers are denoted by $\mathbb N$ and $\mathbb Z_+$ respectively, the set $\{j \in \mathbb N : j \Ge 2\}$ by $\mathbb N_2.$ 
Let $\mathbb R$ stand for the set of real numbers. The function $\ln: (0, \infty) \rar \mathbb R$ denotes the natural logarithm. 
Let $\mu$ be a regular positive Borel measure on the positive real line $(0, \infty),$ and the {\it Laplace transform} $\widehat{\mu}$ of $\mu$ be defined by 
\beqn
\widehat{\mu}(s) = \int_{0}^\infty e^{-s t} \mu(dt), \quad s \in (0, \infty),
\eeqn
whenever the integral converges. 
If $\eta$ is Lebesgue measurable and $\mu$ denotes the weighted Lebesgue measure with weight $\eta,$ then we set $\widehat{\eta}:=\widehat{\mu}.$
We reserve the notation $\lambda$ for the Lebesgue measure on $(0, \infty).$
The Hilbert space of square-summable complex-valued sequences $\{a_j\}_{j = 2}^\infty$ is denoted by $\ell^2(\mathbb N_2).$ Let $\{e_n\}_{n= 2}^\infty$ denote the standard orthonormal basis of $\ell^2(\mathbb N_2).$

Recall from \cite{BPSSV} that the {\it multiplicative Hilbert matrix} is given by 
\beq
 \label{mult-hilb}
\Big(\frac{1}{\sqrt{mn}\ln (mn)}\Big)_{m,n=2}^\infty.
\eeq
Motivated by the study of \eqref{mult-hilb} and related examples (see \cite{PP-1}), we establish norm bounds and study the spectral properties of Helson matrices.
The reader is referred to \cite{BPSSV, MP, PP-1, PP} for the spectral theory of Helson matrices.

This paper is organized as follows. In Section~\ref{S2}, the proof of the boundedness (from above and below) of Helson matrices (Theorem~\ref{l^2_suff_con}) is given, using the Schur-Holmgren-Carleman estimate.  In section~\ref{S1'}, we provide several families of examples illustrating the results of Section 2 as well as those of Proposition~C of the Appendix (Examples~\ref{convex-exp}-\ref{unbb-ex}). 
 In Section~\ref{S4}, criteria are given for the densities of measures to lead to compact and trace class Helson matrices (Theorem~\ref{weight-leb-bdd-comp}). The Section~\ref{S4} further deals with the use of Scattering theory (of the trace-class perturbations) to make some spectral conclusions about a class of Helson matrices (Theorem~\ref{scattering-weight-leb}). Finally, in the Appendix, some characterizations of trace-class and of other Schatten-class properties are discussed (Proposition~A). We conclude this paper by presenting the proofs of the unboundedness (or compactness) of operators given by Helson matrices induced by measures, whose statements appeared earlier in \cite{PP-1}.
 
Though a few of the results presented here (e.g., Proposition~\ref{estimate_norm} and the upper estimate of Theorem~\ref{l^2_suff_con}) were known and stated earlier (\cite{PP-1, PP}), the lower estimate given in Theorem~\ref{l^2_suff_con} is new. While \cite{BPSSV, PP-1, PP} discuss special cases of Lebesgue measure or a special case of Theorem~\ref{l^2_suff_con}(i) in which $\mu$ is a finite measure, the present work treats a more general setting. Furthermore, the methods in \cite{PP-1, PP} differ from those employed here, and the use of a class of signed measures introduces a layer of flexibility from the restrictions of positivity, along with scope of their applications.


\section{Operator norm bounds of Helson matrices} \label{S2}

The main result of this section provides sufficient conditions leading to the lower and upper bounds for the norm of the Helson matrix.

\begin{theorem}\label{l^2_suff_con}
 Let $\{\alpha(n)\}_{n=2}^\infty$ be a sequence of complex numbers. 
Then, \begin{itemize}
	\item[$\mathrm{(i)}$] if, for some $D>0,$
	\beq \label{upper-est}
	|\alpha(n)| \Le \frac{D}{\sqrt{n}\ln n}, \quad n \Ge 2,
	\eeq
	then 
	$\|(\alpha(mn))_{m,n=2}^\infty\| \Le D \pi,$
	\item[$\mathrm{(ii)}$] if alongwith $\mathrm{(i)},$ one also has that $\alpha(n) \in (0,\infty)$ for all $n \Ge 2$ and 
	\beq \label{lower-est}
	\alpha(n) \Ge \frac{C}{\sqrt{n}(b +\ln n)}, \quad n \Ge 2
	\eeq
	for some $b \Ge 0$ and $C>0,$ then $(\alpha(mn))_{m,n=2}^\infty$ is a bounded operator on $\ell^2(\mathbb N_2)$ with $C \pi \Le\|(\alpha(mn))_{m,n=2}^\infty\| \Le D\pi.$
\end{itemize} 
\end{theorem}

 The following recalls the Integral test needed to prove Theorem~\ref{l^2_suff_con}.

{\it Integral estimate \cite[Theorem~9.2.6]{Ba}: Let $f:[1,\infty) \rar [0, \infty]$ be a decreasing function. Then, for any $\alpha \in \mathbb N_2$ and $\beta \in \mathbb N_2 \cup \{\infty\}$ with $\alpha < \beta,$
\beq \label{int-test-est}
\int_{\alpha}^\beta f(t)dt \Le \sum_{k=\alpha}^\beta f(k) \Le \int_{\alpha - 1}^\beta f(t)dt.
\eeq
If $\int\limits_{\alpha-1}^\infty f(t) dt <  \infty,$ then  
the series  $\sum\limits_{k=\alpha}^\infty f(k)$ converges.
}

The next proposition gives the Schur-Holmgren-Carleman estimate, which is required to prove Theorem~\ref{l^2_suff_con}(i) (see \cite[Theorem~5.2]{HS} and \cite[Section~1.4.3]{Ka}). 
Since we could not find any proof of this classical result in this particular form, its proof is included for the sake of completeness.

\begin{proposition} \label{estimate_norm}
Let $(a_{m,n})_{m,n=2}^\infty$ be an infinite matrix of complex numbers such that for some sequence $\{t_n\}_{n = 2}^\infty$ of positive real numbers,
	\beq \label{M-and-tildeM}
	M: = \sup_{n \Ge 2} \frac{1}{t_n}\sum_{m=2}^\infty |a_{n, m}| t_m < \infty ~~\text{and} ~~ \widetilde{M}: =\sup_{m \Ge 2} \frac{1}{t_m}\sum_{n=2}^\infty |a_{n, m}| t_n < \infty.
	\eeq
Then,  $(a_{m,n})_{m,n=2}^\infty$ defines a bounded linear operator on $\ell^2(\mathbb N_2)$ and 
\beq \label{oper-est}
\|(a_{m,n})_{m,n=2}^\infty\| \Le (M\widetilde{M})^{\frac{1}{2}}
\eeq
for every choice of such sequence.
\end{proposition}
\begin{proof}
For $x = \{x_m\}_{m=2}^\infty \in \ell^2(\mathbb N_2),$ by the Cauchy-Schwarz inequality, an application of Fubini's theorem and using \eqref{M-and-tildeM}, we have that 
\beqn
\sum_{n=2}^\infty \Big|\sum_{m=2}^\infty a_{n,m} x_m\Big|^2 &\Le& \sum_{n=2}^\infty \Big(\sum_{m=2}^\infty |a_{n,m}| |x_m|\Big)^2\\ 
&\Le& \sum_{n=2}^\infty \Big(\sum_{m=2}^\infty |a_{n,m}|t_m\Big) \Big(\sum_{m=2}^\infty  \frac{|a_{n, m}| |x_m|^2}{t_m}\Big)\\
 &\Le& M \sum_{n=2}^\infty t_n \Big(\sum_{m=2}^\infty \frac{|a_{n, m}| |x_m|^2}{t_m}\Big)\\
 &=& M \sum_{m=2}^\infty |x_m|^2 \Big(\frac{1}{t_m}\sum_{n=2}^\infty |a_{n, m}| t_n \Big)\\
 &\Le& M \widetilde{M} \|x\|_2^2.
\eeqn
This proves the inequality \eqref{oper-est}. 
\end{proof}
\begin{remark}\label{estimate_norm1}
	In case, $(|a_{m,n}|)_{m,n=2}^\infty$ is a symmetric matrix, it is clear that $M = \widetilde{M}.$ Therefore, the operator norm estimate in Proposition~\ref{estimate_norm} yields $\|(a_{m,n})_{m,n=2}^\infty\| \Le M.$  
\end{remark}

We shall need the following standard result that is derived in \cite[Section~9.5]{Du}, by contour integration in complex analysis.  
\beq \label{integral-fact}
K_\epsilon: = \int_{0}^\infty \frac{t^{-(\frac{1+\epsilon}{2})}}{1+t}dt = \frac{\pi}{\sin(\frac{\pi(1-\epsilon)}{2})}, \quad 0 < \epsilon <1.
\eeq

\begin{proof}[Proof of Theorem~\ref{l^2_suff_con}]

(i) Let $\{t_m\}_{m=2}^\infty$ be given by  
	\beqn
	t_m = \frac{1}{\sqrt{m\ln m}}, \quad m \Ge 2.
	\eeqn
	We may apply Proposition~\ref{estimate_norm} to $H_\alpha: = (\alpha(mn))_{m,n=2}^\infty$ 
	to conclude that $H_\alpha$ is bounded on $\ell^2(\mathbb N_2)$ with $\|H_\alpha\| \Le D \pi.$ Indeed, 
	by the integral test (see \eqref{int-test-est}),
	\beqn 
	 \sqrt{n\ln n}\sum_{m=2}^\infty \frac{|\alpha(mn)|}{\sqrt{m\ln m}}  \notag
	 &\overset{\eqref{upper-est}}\Le& D\sum_{m=2}^\infty \frac{\sqrt{\ln n}}{m\sqrt{\ln m} \ln(mn)}\\\notag 
	 &\Le& D \int_{1}^\infty \frac{\sqrt{\ln n}}{x\sqrt{\ln x} \ln(nx)}dx ~~(\text{letting $y= \ln x$})\\\notag 
	 &=& D\int_{0}^\infty \frac{\sqrt{\ln n}}{ \sqrt{y} (y+\ln n)}dy~~(\text{letting $z= \sqrt{y}$})\\
	 &=&  D\int_{0}^\infty \frac{2\sqrt{\ln n}}{z^2+\ln n}dz =D\pi.
	\eeqn
	Thus, we obtain
	\beqn
	\sup_{n \Ge 2} \sqrt{n\ln n}\sum_{m=2}^\infty \frac{|\alpha(mn)|}{\sqrt{m\ln m}}  \Le D\pi.
	\eeqn
	Since $H_\alpha$ is symmetric, by Remark~\ref{estimate_norm1} and above estimate, it follows that $H_\alpha$ is bounded operator on $\ell^2(\mathbb N_2)$ with $\|H_\alpha\| \Le D\pi.$

	(ii) By part $\mathrm{(i)}$ of this theorem, $H_\alpha$ is bounded on $\ell^2(\mathbb N_2)$ with $\|H_\alpha\| \Le D\pi.$ 
	To see the lower estimate, for $\epsilon > 0,$ let $a_\epsilon = \{a_\epsilon(n)\}_{n=2}^\infty$ be a sequence defined by
	\beqn 
	a_\epsilon(n) = n^{-\frac{1}{2}} (\ln n)^{-\big(\frac{1+\epsilon}{2}\big)}, \quad n \Ge 2.
	\eeqn
	Note that $\|a_\epsilon\|_2^2 = \sum_{m=2}^\infty \frac{1}{m (\ln m)^{1+\epsilon}} < \infty,$ and furthermore by \eqref{int-test-est}, 
	\beqn
	\frac{1}{\epsilon (\ln 2)^\epsilon} = \int_2^\infty \frac{dx}{x (\ln x)^{1+\epsilon}} \Le \|a_\epsilon\|_2^2.
	\eeqn
	Thus, the lower estimate implies that
	\beq \label{lim_at_0-n}
	\lim_{\epsilon \rar 0^+} \|a_\epsilon\|_2 = +\infty.
	\eeq
Furthermore, by \eqref{int-test-est} and letting $t = \frac{b+\ln x}{\ln n},$
	\allowdisplaybreaks
	\beqn
	(H_{\alpha} (a_\epsilon))(n) 
	&=& \sum_{m=2}^\infty \frac{\alpha(mn)(\ln m)^{-(\frac{1+\epsilon}{2})}}{\sqrt{m}}
	\overset{\eqref{lower-est}}\Ge C \sum_{m=2}^\infty \frac{(\ln m)^{-(\frac{1+\epsilon}{2})}}{m\sqrt{n}(b+ \ln(mn))}\\
	&\Ge& \frac{C}{\sqrt{n}} \int_{2}^\infty \frac{(\ln x)^{-(\frac{1+\epsilon}{2})}}{x(b+\ln(nx))}\,dx
	\Ge \frac{C}{\sqrt{n}} \int_{2}^\infty \frac{(b+\ln x)^{-(\frac{1+\epsilon}{2})}}{x(b+\ln(nx))}\,dx\\
	&=& C \frac{(\ln n)^{-(\frac{1+\epsilon}{2})}}{\sqrt{n}} \int_{\frac{b+\ln 2}{\ln n}}^\infty \frac{t^{-(\frac{1+\epsilon}{2})}}{t+1}\,dt,
	\eeqn
	which together with \eqref{integral-fact} leads to  
	\beqn
	(H_{\alpha} (a_\epsilon))(n)  \Ge C \frac{(\ln n)^{-(\frac{1+\epsilon}{2})}}{\sqrt{n}} \Big(K_\epsilon - \int_{0}^{\frac{b+\ln 2}{\ln n}} \frac{t^{-(\frac{1+\epsilon}{2})}}{1+t}dt\Big).
	\eeqn
	Therefore, 
	\beqn
	\sum_{n=2}^\infty \left|(H_{\alpha} (a_\epsilon))(n) \right|^2 \Ge C^2 \sum_{n=2}^\infty \frac{1}{n(\ln n)^{1+\epsilon}} \Big(K_\epsilon - 
	\frac{2}{1-\epsilon} \Big(\frac{b+\ln 2}{\ln n}\Big)^{\frac{1-\epsilon}{2}}\Big)^2,
	\eeqn
	where we used the fact that 
	\beqn
	\int_{0}^{\frac{b+\ln 2}{\ln n}}\frac{t^{-(\frac{1+\epsilon}{2})}}{1+t}dt  &\Le& \int_{0}^{\frac{b+\ln 2}{\ln n}} t^{-(\frac{1+\epsilon}{2})} dt\\
	&=& \frac{2}{1-\epsilon} \Big(\frac{b+\ln 2}{\ln n}\Big)^{\frac{1-\epsilon}{2}}~~\text{with}~\epsilon \in (0,1).
	\eeqn
Thus, we have that
	\beqn
	&&\sum_{n=2}^\infty \left|(H_{\alpha} (a_\epsilon))(n)\right|^2 \\
	&\Ge& 
	C^2 \sum_{n=2}^\infty \frac{1}{n(\ln n)^{1+\epsilon}} \Big(K_\epsilon^2 +  
	\frac{4}{(1-\epsilon)^2} \Big(\frac{b+\ln 2}{\ln n}\Big)^{1-\epsilon} 
	-\frac{4K_\epsilon}{(1-\epsilon)} \Big(\frac{b+\ln 2}{\ln n}\Big)^{\frac{1-\epsilon}{2}}\Big)\\
	&\Ge& C^2\Big(K_\epsilon^2 \|a_\epsilon\|_2^2 + \frac{4(b+\ln 2)^{1-\epsilon}M_1}{(1-\epsilon)^2} - 
	\frac{4K_\epsilon (b+\ln 2)^{1-\epsilon}M_2}{(1-\epsilon)}\Big),
	\eeqn
	where $M_1 = \sum_{n=2}^\infty \frac{1}{n (\ln n)^2} < \infty$ and $M_2 = \sum_{n=2}^\infty \frac{1}{n (\ln n)^{3/2}} < \infty.$
	Therefore, for $\epsilon \in (0,1),$
	\beqn
	\frac{\|H_\alpha a_\epsilon\|_2^2}{\|a_\epsilon\|_2^2}
	\Ge C^2\Big(K_\epsilon^2 + \frac{4(b+\ln 2)^{1-\epsilon}M_1}{(1-\epsilon)^2\|a_\epsilon\|_2^2}  - 
	\frac{4K_\epsilon (b+\ln 2)^{1-\epsilon}M_2}{(1-\epsilon)\|a_\epsilon\|_2^2} \Big).  
	\eeqn
	Finally, by \eqref{integral-fact} and \eqref{lim_at_0-n}, it follows that
	\beqn
	\|H_{\alpha}\|^2	
	\Ge 
	C^2 \limsup_{\epsilon \rar 0^{+}} K_\epsilon^2
	\overset{\eqref{integral-fact}}= (C\pi)^2,
	\eeqn
	completing the proof.  
\end{proof}

Motivated by the multiplicative Hilbert matrix given in \eqref{mult-hilb} and related examples (see \cite{PP-1}), we introduce below the family of Helson matrices induced by the Laplace transform of positive measures. Before we define it, recall that for a Borel $\sigma$-algebra $\mathcal B((0, \infty)),$ let $\mu: \mathcal B((0, \infty)) \rar \mathbb R$ be a signed regular Borel measure. By the Hahn decomposition theorem (see \cite[Section~17.2]{RF}), 
\beq \label{measure-decom}
\mu(\Delta) = \mu_+(\Delta) - \mu_-(\Delta), \quad \Delta \in \mathcal B((0, \infty)),
\eeq
where $\mu_{+}$ and $\mu_{-}$ are the positive and negative parts of $\mu$ having disjoint support. 
The {\it total variation measure} $|\mu|(\cdot)$ of a measure $\mu$ is given by 
\beqn
|\mu|(\Delta) = \mu_+(\Delta) + \mu_-(\Delta), \quad \Delta \in \mathcal B((0, \infty)).
\eeqn   
\begin{definition}\label{H-mu}
	Let $\mu$ be a signed regular Borel measure on $(0, \infty)$ such that
	$\int_{0}^\infty 2^{-t} |\mu|(dt) < \infty.$
	The {\it Helson matrix $H_\mu$ induced by the measure $\mu,$} with respect to the standard basis $\{e_n\}_{n=2}^\infty$ in $\ell^2(\mathbb N_2),$ is given by 
	\beqn 
	H_\mu = \Big(\frac{\widehat{\mu} (\ln(mn))}{\sqrt{mn}}\Big)_{m,n=2}^\infty. 
	\eeqn
\end{definition}

\begin{remark}
(i):	It is worthwhile noting that by the virtue of condition on the measure $\mu$ and by an application of Lebesgue dominated convergence theorem, it follows easily that $\widehat{\mu}(\ln n) \rar 0$ as $n \rar \infty.$

(ii): It turns out that $H_\mu$ is a ``superposition'' of the Helson matrices $H_{\delta_c}$ of the Dirac delta measures $\delta_c,$ $c > 0.$ 
This simply means that for a positive measure $\mu$ on $(0, \infty)$, if $H_{\mu}$ defines a bounded operator on $\ell^2(\mathbb N_2),$ 
then
	\beqn
	\inp{{H}_\mu x}{y} = \int_0^\infty  \inp{{H}_{\delta_t} x}{y}  \mu(dt), \quad x, y \in \ell^2(\mathbb N_2).
	\eeqn
Indeed, note that for $x=\{x_m\}_{m=2}^\infty$ and $y=\{y_m\}_{m=2}^\infty$ in $\ell^2(\mathbb N_2),$
\beqn
\sum_{m, n=2}^{\infty}  \frac{\widehat{\mu}(\ln(mn))}{\sqrt{mn}}|x_m| |y_n| =  \inp{{H}_\mu |x|}{|y|} \Le \|H_\mu\| \|x\| \|y\| < \infty,
\eeqn
where $|z|=\{|z_m|\}_{m=2}^\infty$ for $z=\{z_m\}_{m=2}^\infty.$ It now follows from Fubini's theorem that
\beqn
\inp{{H}_\mu x}{y} &=&  \int_0^{\infty} \Big(\sum_{m=2}^{\infty} x_m m^{-t-\frac{1}{2}} \Big) \Big(\sum_{n=2}^{\infty}n^{-t-\frac{1}{2}}\, \overline{y_n}\Big)\mu(dt), \\
\eeqn
which is same as $\int\limits_0^\infty  \inp{{H}_{\delta_t} x}{y}  \mu(dt).$ 
\end{remark}

As a consequence of Theorem~\ref{l^2_suff_con}, the next result provides growth conditions in terms of the
Laplace transform $\hat{\mu}$ leading to the lower and upper bounds for the norm of $H_\mu$.


\begin{corollary} \label{trace-prop-gen}
	Let $\mu: \mathcal B((0, \infty)) \rar \mathbb R$ be a signed measure such that $\int_0^\infty 2^{-t}|\mu|(dt) < \infty$. Suppose that there exists $D>0$	such that
	\beq
	\label{meas-bdd-cond}
	|\widehat{\mu}(\ln n)| \Le \frac{D}{\ln n}, \quad n \Ge 2.
	\eeq
	Then, $H_\mu$ is a bounded self-adjoint operator on $\ell^2(\mathbb N_2)$. Furthermore, if $\mu$ is positive on $(0, \infty)$, and there exist $b, C >0$ such that   
	\beqn
	\widehat{\mu}(\ln n) \Ge \frac{C}{b + \ln n}, \quad n \Ge 2,
	\eeqn
	then $H_\mu$ is bounded with $\|H_\mu\| \in [C\pi, D\pi]$.
\end{corollary}
\begin{proof}
	Observe that for the matrix $H_\mu,$ by \eqref{meas-bdd-cond}, the condition \eqref{upper-est} holds. Thus, by Theorem~\ref{l^2_suff_con}(i), $H_\mu$ is bounded self-adjoint operator. The remaining part is also immediate from Theorem~\ref{l^2_suff_con}(ii).
\end{proof}

The next result gives sufficient conditions ensuring that both $H_{|\mu|}$ and $H_\mu$ are (possibly unbounded) self-adjoint operators.

\begin{proposition}\label{H-mu-decom}
	Let $\mu$ be a signed measure on $(0, \infty)$ as in Definition~\ref{H-mu} and let $\mu_{\pm}$ be associated positive measures in \eqref{measure-decom}. Denote by $\mathcal N_\pm \equiv \mathcal N_{\mu_\pm}$ the closed, densely defined operators respectively as in Lemma~B. Then
	\begin{itemize}
		\item [(i)] $H_{\mu_\pm} = \mathcal N^*_{\pm}\mathcal N_{\pm}$ are positive self-adjoint operators $($possibly unbounded$)$ in 
		$\ell^2(\mathbb N_2)$,
		\item [(ii)] if $\mathcal D(\mathcal N_-) \subseteq \mathcal D(\mathcal N_{+})$ and if there exists $0 \Le a < 1$ and $b > 0$ such that
	    \beq \label{suff-inequality}
		\|\mathcal N_+ f\|^2 \Le a \|\mathcal N_- f\|^2 + b \|f\|^2, \quad f \in \mathcal D(\mathcal N_-),
		\eeq
		then one can associate two unique self-adjoint operators $H_{|\mu|}$ and $H_\mu$ such that
	$$
		H_{\mu_+} + H_{\mu_-} = H_{|\mu|} ~\text{and}~H_{\mu_+} - H_{\mu_-} =H_\mu,
		$$
		in the sense of forms on $\mathcal D(H_{|\mu|})$ and $\mathcal D(H_\mu) \subseteq \mathcal D(\mathcal N_{-})$, respectively,
 \item [(iii)] if $\mathcal D(\mathcal N_{+}) \subseteq \mathcal D(\mathcal N_{-})$ and if there exists $0 \Le a < 1$ and $b > 0$ such that
 \beqn
\|\mathcal N_- f\|^2 \Le a \|\mathcal N_+ f\|^2 + b \|f\|^2, \quad f \in \mathcal D(\mathcal N_{+}),
\eeqn
then one can associate two unique self-adjoint operators $H_{|\mu|}$ and $H_\mu$ such that
$
H_{\mu_+} + H_{\mu_-} = H_{|\mu|} ~\text{and}~H_{\mu_+} - H_{\mu_-} =H_\mu,
$
in the sense of forms on $\mathcal D(H_{|\mu|})$ and $\mathcal D(H_\mu) \subseteq \mathcal D(\mathcal N_{+})$, respectively,
		\item [(iv)] if one of $\mathcal N_\pm$ is bounded  on $\ell^2(\mathbb N_2)$, then $H_\mu = H_{\mu_+} - H_{\mu_-}$ and $H_{|\mu|} = H_{\mu_+} + H_{\mu_-}$, as sum of operators on the smaller domain $\mathcal D(H_{\mu_\mp}) \subseteq \mathcal D(\mathcal N_{\mp}),$ respectively.
	\end{itemize}
\end{proposition}
\begin{proof}
The part (i) is a direct consequence of Lemma~B in the Appendix. For (ii), let $\|\mathcal N_+ f\|^2 \Le a_+ \|\mathcal N_- f\|^2 + b_+ \|f\|^2$ for all $f \in \mathcal D(\mathcal N_-) \subseteq \mathcal D(\mathcal N_+)$ with $0 \Le a_+ < 1$ and $b_+ > 0$. This means that the form-domain of $H_{\mu_-}$ is contained in the form-domain of $H_{\mu_+}$ and that on the smaller domain $\mathcal D(\mathcal N_-),$ the form of $H_{\mu_+}$ is bounded relative to that of $H_{\mu_-}$ with relative form-bound being less than $1$ (\cite[pp.~319]{Ka}, \cite[pp.~167--169]{RS-2}). Therefore, by \cite[Theorem~X.17]{RS-2}, there exists a unique pair of self-adjoint operators, denoted by $H_\mu = - \widetilde{H}_{\mu}$ and $H_{|\mu|}$ respectively such that 
	\beqn
	\inp{f}{\widetilde{H}_{{\mu}} f} = \inp{f}{H_{\mu_-}f} -\inp{f}{H_{\mu_+} f}, \quad
f \in \mathcal D(H_\mu) \subseteq \mathcal D(\mathcal N_-), \\
\inp{g}{H_{|\mu|} g} = \inp{g}{H_{\mu_+} g} +\inp{g}{H_{\mu_-}g}, \quad g \in \mathcal D(H_{|\mu|}) \subseteq \mathcal D(\mathcal N_-),
	\eeqn
	proving (ii).
The proof of part (iii) follows the same argument as in (ii), and hence we omit the details.
 	
(iv) If say, $\mathcal N_+$ is bounded on $\ell^2(\mathbb N_2)$, then the inequality \eqref{suff-inequality} is satisfied with $a_+ = 0$, and hence $H_\mu = \mathcal N_+^*\mathcal N_+ - H_{\mu_-} = \mathcal N_+^*\mathcal N_+ - \mathcal N_-^*\mathcal N_-$, not only in the sense of form-sum, but also as a sum of operators on $\mathcal D(H_{\mu_-}) = \mathcal D(\mathcal N_-^*\mathcal N_-)$.
\end{proof}

As far as we know, the following consequence of Theorem~\ref{l^2_suff_con} has not been recorded in the existing literature.  
\begin{corollary}\label{convex-measure}
	Let $\mu_j~ (j=1, 2, \ldots, n)$ be positive measures on $(0, \infty)$ as in Definition~\ref{H-mu}. Let $\nu$ be a convex combination of $\mu_1, \mu_2, \ldots \mu_n,$ i.e., 
	\beq \label{convex-comb-measure}
	\nu(\Delta) = \sum_{j=1}^n t_j \mu_j(\Delta) ~ \text{for all Borel sets $\Delta \subseteq (0, \infty)$}
	\eeq 
	with $t_j \in [0,1]$ for all $j=1, \ldots, n$ and $\sum_{j=1}^n t_j = 1.$
	If, for $j=1, \ldots, n,$ there exist $b_j \Ge 0$ and constants $C_j, D_j > 0$ such that
	\beq \label{upp-low-bound-2}
		\frac{C_j}{b_j +\ln n} \Le \widehat{\mu}_j(\ln n) \Le \frac{D_j}{\ln n}, \quad n \Ge 2,
	\eeq
	then $H_{\nu}$ is bounded on $\ell^2(\mathbb N_2)$ with $C\pi \Le \|H_{\nu}\| \Le D\pi$, where $C= \min\{C_j : j =1, \ldots, n\}$ and $D = \max\{D_j: j=1, \ldots, n\}.$  
\end{corollary}
\begin{proof}
	Note that $\int_0^\infty 2^{-t}\nu(dt) < \infty$ and $\widehat{\nu}(t) = \sum_{j=1}^n t_j\widehat{\mu}_j(t)$ for all $t \Ge \ln 2.$ If we set $b=\max\{b_1, \ldots, b_n\},$ then by \eqref{upp-low-bound-2},
	\beqn
	 \frac{C}{b + \ln n} \Le \widehat{\nu}(\ln n)
	 \Le \frac{D}{\ln n},~~ n \Ge 2.
	\eeqn
This together with Corollary~\ref{trace-prop-gen} completes the proof.
\end{proof}

\section{Examples}\label{S1'}
In this section, we provide several examples illustrating the results established in this paper.
An application of Corollary~\ref{convex-measure} provides  the precise norm of Helson matrices induced by a convex combination of exponential measures.
\begin{example}\label{convex-exp}
Let $\mu_j(dt) = e^{-a_jt}dt,$ where $a_j \Ge 0 ~(j=1,2),$ and let $\nu_r$ be given by \eqref{convex-comb-measure}. Then, 
\beqn
\frac{1}{\max\{a_1, a_2\} + \ln n} \Le  \widehat{\nu}_r(\ln n) \Le \frac{1}{\ln n}, \quad n \Ge 2,
\eeqn
and hence, by Corollary~\ref{convex-measure}, $H_{\nu_r}$ is bounded on $\ell^2(\mathbb N_2)$ with $\|H_{\nu_r}\| = \pi,$ independent of any value of $r \in [0,1].$ 
As special sub-cases,
\begin{itemize}
\item[$\mathrm{(i)}$] if $\mu_a(dt) = e^{-at}dt$ with $a \Ge 0,$ then $H_a \equiv H_{\mu_a}$ is bounded with $\|H_{a}\| = \pi$ for all $a >0.$ However, it should be noted that while $H_{a}$ here acts on $\ell^2(\mathbb N_2),$ a similar matrix 
		\beqn
		\widetilde{H}_{a} = \Big(\frac{1}{\sqrt{mn}(a+\ln(mn))}\Big)_{m,n=1}^\infty
		\eeqn
		has been considered in \cite{PP-1}. There, it has been shown that $\widetilde{H}_{a}$ is bounded on $\ell^2(\mathbb N)$ with operator norm, can be greater than $\pi$ unlike here $\|H_a\| = \pi$ for all $a >0.$ To understand this apparent contradiction, let $\ell^2(\mathbb N_2)$ be embedded as a subspace of $\ell^2(\mathbb N),$ consider $x = \{x_n\}_{n=1}^\infty$ be such that $x_1 \neq 0$ and $x_n = 0,~n \Ge 2.$ Then, 
		\beqn
		\inp{\widetilde{H}_{a}x}{x} = \frac{|x_1|^2}{a} = \frac{\|x\|_2^2}{a},
		\eeqn
		which implies that $\|\widetilde{H}_{a}\| \Ge \frac{1}{a}.$ Therefore, if $a \in (0, \pi^{-1}),$ then $\|\widetilde{H}_{a}\| > \pi,$ giving that the operator norm is beyond $\pi$ for certain values of $a.$ However, for the matrix $H_{a}$ (induced by $e^{-at}dt$) in $\ell^2(\mathbb N_2),$ the absolutely continuous spectrum is $[0,\pi]$ and $\|H_a\| = \pi,$ independent of any $a > 0.$
	\item[$\mathrm{(ii)}$] if we consider $\mu(dt) = e^{-at}\cosh(\omega t)dt$ with $\omega \in \mathbb R$ satisfying $a \Ge |\omega|,$ then $\mu$ may be rewritten as  $(e^{-(a-\omega)t} + e^{-(a+\omega)t})dt/2,$ and hence $\|H_\mu\| = \pi.$  
	\hfill $\diamondsuit$ 
\end{itemize} 
\end{example} 

\begin{example} \label{Seip-int-exam}
(i) 	For $\alpha \Ge 0,$ let $\eta_\alpha(t) = \frac{1}{(1+t)^\alpha},~ t >  0$ and let $\mu_\alpha(dt) = \eta_\alpha(t)dt.$ 
		Then, $e^{-\alpha t} \Le \eta_\alpha(t) \Le 1$ on $(0, \infty),$ which implies that
		\beqn 
		\frac{1}{\alpha + \ln n} = \int_{0}^\infty n^{-t} e^{-\alpha t}dt \Le \widehat{\eta}_\alpha(\ln n) 
		\Le \int_{0}^\infty n^{-t} dt = \frac{1}{\ln n}, \quad n \Ge 2.
		\eeqn
		Hence, $\|H_{\mu_\alpha}\| = \pi$ for every $\alpha \Ge 0,$ since $C=D=1$ in this example.
	
	(ii) 	For $a, c, p \Ge 0,$ consider $\eta(t) = t^p + ce^{-at}, t > 0.$  Then,
		\beq \label{weight-t^p}
		\widehat{\eta}(\ln n)  = 
		\int_{0}^\infty  n^{-t}(t^p + ce^{-at}) dt
		= \frac{\Gamma(p+1)}{(\ln n)^{p+1}} + \frac{c}{a+\ln n}, \quad n \Ge 2,
		\eeq
		which implies that 	
		\beqn
		\frac{c}{a+\ln n} \Le \widehat{\eta}(\ln n) \Le \Big(\frac{\Gamma(1+p)}{(\ln 2)^{p}}+c\Big) \frac{1}{\ln n}, \quad n \Ge 2,
		\eeqn 
		where $\Gamma$ denotes the Gamma function.
		Hence, the corresponding $H_{\mu}$ is bounded on $\ell^2(\mathbb N_2)$ with 
		$
		c \pi \Le \|H_{\mu}\| \Le \big(\frac{\Gamma(1+p)}{(\ln 2)^{p}}+c\big)\pi.
		$
	\hfill $\diamondsuit$
\end{example}  

The following example is one for which the operator norm of $H_\mu$ is strictly bigger than $\pi$ unlike many examples discussed above. 

\begin{example}
	Let $\mu$ be a positive measure such that $H_\mu$ is bounded on $\ell^2(\mathbb N_2).$ An obvious lower bound of $H_\mu$ is given by
	\beqn
	\|H_\mu\| \Ge \inp{H_\mu e_2}{e_2} = \frac{\widehat{\mu}(\ln 4)}{2}.
	\eeqn 
	In particular, if $\eta_p(t) =t^p,~ t > 0$ for $p>0,$ then,
	by using \eqref{weight-t^p}, 
	\beq \label{lower-bound-t^p}
	\|{H}_{\mu_p}\| \Ge \frac{\Gamma(p+1)}{2(\ln 4)^{p+1}}.
	\eeq
	Let $h(p): = \frac{\Gamma(p+1)}{2 (\ln 4)^{p+1}}, ~p > 0.$
Since $\frac{h(p+1)}{h(p)} = \frac{p+1}{\ln 4}$ for all $p \Ge 1$, it follows that $h$ is an increasing function on $\mathbb N_2$. 
	Hence, there exists a positive integer $p_0$ such that 
	\beqn
	\frac{\Gamma(p+1)}{2 (\ln 4)^{p+1}} > \pi~ \text{for all integers}~p \Ge p_0.
	\eeqn
	To obtain the least integer $p_0$ satisfying this, observe that $h(5) \approx 8.44,$ which is bigger than $\pi.$ Also, since $h(4) \approx 2.43 < \pi,$ the least value of $p_0$ is exactly $5.$ This together with \eqref{lower-bound-t^p} yields
	\beqn
	\|{H}_{\mu_p}\| > \pi~\mbox{for all integers}~ p \Ge 5.
	\eeqn 
	Moreover, since $(\ln n)\widehat{\mu}_p(\ln n) \rar 0$ as $n \rar \infty,$ by Proposition~C(ii) of the Appendix, ${H}_{\mu_p}$ is compact on $\ell^2(\mathbb N_2),$ and as we shall see later in Example~\ref{exam-weight-leb}(i) that it is in fact a trace-class operator. 
	\hfill $\diamondsuit$
\end{example}

\begin{example}  
	For $p>0,$ let $\eta(t)= t^p\sin(t).$ Then, $\mu$ is a non-positive infinite measure and $\mu(dt) = \mu_+(dt) - \mu_{-}(dt),$ where $\mu_{\pm}$ are positive measures given by $\mu_{\pm}(dt) = t^p \max\{\pm \sin(t), 0\}(dt).$ 
Thus, $$\widehat{\mu}_{\pm}(x) \Le \int_0^\infty e^{-xt}t^pdt = \frac{\Gamma(p+1)}{x^{p+1}},$$
	so that the operators $H_{\mu_{\pm}}$, the matrix elements of which are given by $(\frac{\widehat{\mu}_{\pm}(\ln mn)}{\sqrt{mn}})_{m,n=2}^\infty$, satisfy the bound $\frac{\widehat{\mu}_\pm(\ln n)}{\sqrt{n}} \Le \frac{\Gamma(p+1)}{\sqrt{n}(\ln n)^{p+1}}, n \Ge 2$.
	Therefore, by Theorem~\ref{upper-est}(i), $H_{\mu_{\pm}}$ is bounded on $\ell^2(\mathbb N_2)$. Furthermore, since $H_{\mu_{\pm}}$ are both positive and since
	$$
	\sum_{m=2}^\infty \frac{\widehat{\mu}_\pm(2 \ln m)}{m} \Le \frac{\Gamma(p+1)}{2^{p+1}} \sum_{m=2}^\infty \frac{1}{m (\ln m)^{p+1}} < \infty,
	$$
	it follows by Proposition~A(i) of the Appendix that $H_{\mu_{\pm}}$ are trace-class. This together with  Proposition~\ref{H-mu-decom} implies that $H_\mu = H_{\mu_+} - H_{\mu_-}$ is also trace-class. 
	\hfill $\diamondsuit$
\end{example}

The next example yields a class of unbounded self-adjoint Helson matrices.

\begin{example}\label{unbb-ex}
	For $p \in (0,1),$ consider the weighted Lebesgue measure $\mu_p$ with weight function $\eta_p(t) = t^{-p},~t > 0.$ Then, $\widehat{\eta}_p(t) = \Gamma(1-p)t^{p-1} (t > 0),$ which gives that $(\ln n) \widehat{\eta}_p(\ln n) =\Gamma(1-p)(\ln n)^{p} \rar\infty$ as $n \rar \infty.$ Hence, by Proposition~C(i) of the Appendix, $H_{\mu_p}$ is an unbounded positive self-adjoint operator. \hfill $\diamondsuit$
\end{example}

\begin{example}
For $p \in (0,1)$ and $a > 0$, consider the weighted Lebesgue measure $\mu$ with weight function $\eta(t) = t^{-p}- e^{-at},~t > 0.$ 
	In this case, $\mu_+(dt) = t^{-p}dt$ and $\mu_-(dt) = e^{-at}dt.$  By Example~\ref{unbb-ex}, $H_{\mu_+}$ is unbounded in $\ell^2(\mathbb N_2).$ On the other hand, by Example~\ref{convex-exp}(i), $H_{\mu_-}$ is bounded in $\ell^2(\mathbb N_2)$. Thus, 
	$$\|\mathcal N_- f\|^2 = \inp{\mathcal N_-^* \mathcal N_- f}{f} = \inp{H_{\mu_-}f}{f} \Le \|H_{\mu_-}\| \|f\|^2, \quad f \in \ell^2(\mathbb N_2),$$ which implies that $\mathcal N_-$ is bounded. Therefore, by Proposition~\ref{H-mu-decom}(iv), $H_\mu = H_{\mu_+}  - H_{\mu_-}$ and $H_{|\mu|} = H_{\mu_+} + H_{\mu_-}$, as sum of operators on $\mathcal D(H_{\mu_+}).$
\end{example}

\section{Scattering between two Helson matrices}\label{S4}

For definitions of various spectra of a self-adjoint operator (e.g., essential spectrum $\sigma_{ess}(\cdot),$ absolutely continuous spectrum $\sigma_{ac}(\cdot))$, the reader is referred to \cite{RS-1, Si-2}.

The following result is devoted to studying the Helson matrices induced by weighted Lebesgue measures.

\begin{theorem}\label{weight-leb-bdd-comp}
Let $\mu$ be a weighted Lebesgue measure on $(0, \infty)$ with a continuous positive weight function $\eta$ satisfying $\int_{0}^\infty 2^{-y} \eta(y) dy < \infty$. Suppose furthermore that $\eta$ is monotone on $(0, \infty)$ with finite right limit $\eta(0_+)$. 
	\begin{itemize}
		\item[$\mathrm{(i)}$] If $\eta(t) \Ge \eta(0_+)$ for all $t>0$ $($i.e., $\eta$ is non-decreasing$)$, then $x\,\widehat{\eta}(x) \Ge \eta(0_+)$ and $x\,\widehat{\eta}(x) \rar \eta(0_+)$ as $x \rar \infty$. On the other hand, if $\eta(t) \Le \eta(0_+)$ for all $t>0$ $($i.e., $\eta$ is non-increasing$)$, then $x\,\widehat{\eta}(x) \Le \eta(0_+)$ and $x\,\widehat{\eta}(x) \rar \eta(0_+)$ as $x \rar \infty$, and in both these cases,
	    $$\|H_\mu\| \Le \pi \big(\sup_{x \Ge \ln 2} x \widehat{\eta}(x)\big).$$ Moreover, $H_\mu - \eta(0_+) H_\lambda$ $($or $\eta(0_+) H_\lambda - H_\mu$ respectively$)$ is positive, self-adjoint, compact, and 
	    $\sigma_{ess}(H_\mu) =  [0, \eta(0_+)\pi]$.
		\item[$\mathrm{(ii)}$] if in either of the cases above, instead, for some $p >0,$ there exist $D \equiv D(p)>0$ and a positive measurable function $g$ on $(0, \infty)$ such that $\int_0^\infty y^pg(y)e^{-y}dy < \infty$ and such that
		\beq \label{eta-ineq-2}
		\Big|\eta\Big(\frac{y}{x}\Big) - \eta(0_+)\Big| \Le D \Big(\frac{y}{x}\Big)^pg(y), \quad y > 0,~x \Ge \ln 2,
		\eeq
		then $H_\mu -\eta(0_+)H_{\lambda}$ $($or $\eta(0_+)H_{\lambda}-H_\mu$ respectively$)$ is positive, self-adjoint and trace-class.
	\end{itemize}
\end{theorem}
\begin{proof}
	For any $x \Ge \ln 2,$ 
	\beqn
	x\,\widehat{\eta}(x) = x\int_0^\infty e^{-tx}\eta(t)dt = \int_0^\infty \eta\Big(\frac{y}{x}\Big) e^{-y}dy
	\eeqn
	leading to 
	\beq \label{lap-diff}
	x\, \widehat{\eta}(x) - \eta(0_+) = \int_0^\infty \Big(\eta\Big(\frac{y}{x}\Big) - \eta(0_+)\Big)e^{-y}dy, 
	\eeq
	where we used the fact that $\int_0^\infty e^{-y}dy = 1.$
	
	(i) Since $\eta$ is continuous on $(0, \infty)$ with finite right limit at $0$, $\eta\big(\frac{y}{x}\big)$ converges to $\eta(0_+)$ as $x \rar \infty$ for each $y > 0$. We first consider the case in which $\eta$ is non-decreasing. In this case,  for all $x \Ge \ln 2$ and $y >0$, we have $$\eta\big(\frac{y}{x}\big) \Le \eta\big(\frac{y}{\ln 2}\big)~~\text{and}~~\eta(0_+) \Le \eta\big(\frac{y}{\ln 2}\big).$$  This combined with \eqref{lap-diff} and the triangle inequality yields that  
	\beqn
	0 \Le x \widehat{\eta}(x) - \eta(0_+) \Le \int_0^\infty \Big|\eta\big(\frac{y}{x}\big) - \eta(0_+)\Big|e^{-y} dy &\Le& 2 \int_0^\infty \eta\big(\frac{y}{\ln 2}\big) e^{-y} dy\\ 
	&=&  \ln 4 \int_0^\infty 2^{-y} \eta(y) dy < \infty.
	\eeqn
	Hence, by an application of the dominated convergence theorem,  $x \widehat{\eta}(x) \rightarrow \eta(0_+)$ as $x \rar \infty$. 
	On the other hand, if $\eta$ is non-increasing, then  $$\Big|\eta\big(\frac{y}{x}\big) - \eta(0_+)\Big| \Le 2 \eta(0_+), \quad x \Ge \ln 2, ~y > 0,$$ and hence by \eqref{lap-diff},
	$$0  \Le \eta(0_+) - x \widehat{\eta}(x) \Le 2\eta(0_+) \int_0^\infty  e^{-y} dy = 2\eta(0_+), \quad x \Ge \ln 2.$$ Again as above, by the dominated convergence theorem, $x \widehat{\eta}(x)  \rar \eta(0_+)$ as $x \rar \infty$. Therefore, $\widehat{\eta}(x) = O(\frac{1}{x})$ as $x \rar \infty$, which combined with Theorem~\ref{l^2_suff_con}(i) yields that $\|H_\mu\| \Le \pi \big(\sup_{x \Ge \ln 2} x \widehat{\eta}(x)\big)$.
	
	
Since the Laplace transform of Lebesgue measure on $(0,\infty)$ is $1/x$, the above discussion yields that $x\big(\widehat{\eta}(x) - \eta(0_+)\widehat{\lambda}(x)\big)$ is non-negative and tends to $0$ as $x \rar \infty$ (or $x \big(\eta(0_+)\widehat{\lambda}(x) - \widehat{\eta}(x)\big)$ is non-negative and tends to $0$ as $x \rar \infty$, respectively). Thus, by Proposition~C(ii) in the Appendix, 
    $H_\mu- \eta(0_+)H_{\lambda}$ (or $\eta(0_+)H_{\lambda}-H_\mu$, respectively) is compact, positive, self-adjoint. Therefore, it follows from Weyl's theorem \cite[Theorem~IV.5.35]{Ka} and \cite[Theorem~1]{BPSSV}, 
	\beqn
	\sigma_{ess}(H_\mu) = \sigma_{ess}(\eta(0_+)H_{\lambda}) = 
	[0, \eta(0_+)\pi],
	\eeqn
    proving (i).
	
(ii) Note that by \eqref{lap-diff},
	\beqn
	|x \widehat{\eta}(x) - \eta(0_+)| 
	\overset{\eqref{eta-ineq-2}}\Le 
	D \int_0^\infty \Big(\frac{y}{x}\Big)^p g(y)e^{-y}dy
	&=& \widetilde{D}x^{-p}, \quad x \Ge \ln 2, 
	\eeqn
	where $\widetilde{D} = D\int_0^\infty y^p g(y)e^{-y}dy.$
	Since the
	Laplace transform of Lebesgue measure on $(0, \infty)$ is $\frac{1}{x}$, as in (i),  
	\beqn
	x\big(\widehat{\eta}(x) - \eta(0_+)\widehat{\lambda}(x)\big) \Le \widetilde{D}x^{-p},\quad p > 0, \quad x \Ge \ln 2,
	\eeqn
	$($or $x\big(\eta(0_+)\widehat{\lambda}(x)- \widehat{\eta}(x)\big) \Le  \widetilde{D}x^{-p}$ respectively$).$
	If $\eta(t) \Ge \eta(0_+)$ on $(0, \infty),$ then  $\widehat{\eta} -\eta(0_+)\widehat{\lambda}$ is a positive function, and 
	\beqn
	\big(\widehat{\eta} - \eta(0_+)\widehat{\lambda}\big)(\ln n) \Le \frac{\widetilde{D}}{(\ln n)^{p+1}}, \quad n \Ge 2,
	\eeqn
	which implies that
	\beqn
	\sum_{n=2}^\infty \frac{\big(\widehat{\eta} - \eta(0_+)\widehat{\lambda}\big)(2\ln n)}{n} \Le \widetilde{D} \sum_{n=2}^\infty \frac{1}{n (2\ln n)^{p+1}}  < \infty.
	\eeqn
	Thus, by Proposition~A(i) of the Appendix, $H_\mu-\eta(0_+)H_{\lambda}$ is positive and of trace-class
	(or $\eta(0_+)H_{\lambda}- H_\mu$ is positive and of trace-class respectively).
\end{proof}

\begin{remark}
If $\eta(0_+) = 0$ (in which case only the non-decreasing case survives), then $H_\mu$ itself is compact, positive, self-adjoint and obviously, $\sigma_{ess}(H_\mu) = \{0\}$.
\end{remark}

Here are some special instances of Theorem~\ref{weight-leb-bdd-comp}.
\begin{example}\label{exam-weight-leb}
	(i) For $\eta(t) = t^p, \eta(0_+) = 0$ and 
	\beqn
	\Big|\eta\Big(\frac{y}{x}\Big) - \eta(0_+)\Big| = \Big(\frac{y}{x}\Big)^p.
	\eeqn
	Applying Theorem~\ref{weight-leb-bdd-comp}  with $g = 1$ gives that $H_\mu$ is positive and of trace-class. 
	
	(ii) For $\eta(t) = \ln (1+t),$ $\eta(0_+) = 0$ and 
	\beqn
	\Big|\eta\Big(\frac{y}{x}\Big) - \eta(0_+)\Big| = \ln \Big(1+\frac{y}{x}\Big) \Le \frac{y}{x}, \quad y,~ x > 0.
	\eeqn
	Therefore, $H_\mu$ is positive and of trace-class.
	
	(iii) For $\eta(t) = e^{at^p}$ with $a>0$ and $p \in (0,1),$ we have that $\eta(0_+) =1$ and 
	\beqn
	\eta\Big(\frac{y}{x}\Big) - 1 = a\Big(\frac{y}{x}\Big)^p \sum_{k=1}^\infty \frac{\Big(a\Big(\frac{y}{x}\Big)^p\Big)^{k}}{(k+1)!} \Le a \Big(\frac{y}{x}\Big)^p e^{a\big(\frac{y}{\ln 2}\big)^p}, \quad x \Ge \ln 2,
	\eeqn
	and hence  
	\beqn
	\int_0^\infty y^p e^{a\big(\frac{y}{\ln 2}\big)^p}e^{-y}dy < \infty,
	\eeqn
	since $p<1,$ and thus by Theorem~\ref{weight-leb-bdd-comp}(ii), $H_\mu - H_{\lambda}$ is of trace-class.
	%
	
	(iv) For a non-negative continuous function $h$ on $(0, \infty),$ let $\mu$ be the positive measure given by $\mu(dt)= e^{-h(t)}dt.$  
	Assume that $h$ increases (or decreases respectively) to the right limit $h(0_+),$ and satisfies 
	\beqn
	h\Big(\frac{y}{x}\Big) - h(0_+) \Le D \Big(\frac{y}{x}\Big)^p~ \text{for $p>0$}
	\eeqn
	(or $h(0_+)-h\Big(\frac{y}{x}\Big) \Le D \Big(\frac{y}{x}\Big)^p$ respectively). Then, the corresponding $\eta$ decreases (or increases respectively), and
	\beqn
	0 < \eta(0_+) - \eta\Big(\frac{y}{x}\Big) \Le h\Big(\frac{y}{x}\Big) – h(0_+) \Le D\Big(\frac{y}{x}\Big)^{p}.
	\eeqn 
	Since $\int_0^\infty y^pe^{-y}dy < \infty,$ by Theorem~\ref{weight-leb-bdd-comp}(ii), it follows that $e^{-h(0_+)}H_{\lambda}-H_\mu$ (or $H_\mu-e^{-h(0_+)}H_{\lambda}$) is positive and of trace-class.
	As a special case, one may take $h(t) = at$ with $a > 0,$ then $H_{\lambda}-H_\mu$ is positive and of trace-class.
	
	(v) For $m \in \mathbb N,$ let $\eta_m(t) = (1+t)^m.$ Then, $\eta_m(0_+) =1$ and for $x \Ge \ln 2,$
	\beqn
	\eta_m\Big(\frac{y}{x}\Big)-1 = \sum_{k=1}^m {m\choose k} \Big(\frac{y}{x}\Big)^k &\Le& (\ln 2)^{1-m}\Big(\frac{y}{x}\Big) \sum_{k=1}^m {m\choose k}y^{k-1}.
	\eeqn
	If $g_m(y) = \sum_{k=1}^m {m\choose k}y^{k-1},$ then $\int_0^\infty yg_m(y)e^{-y}dy < \infty,$ and consequently,  $H_{\mu_m} - H_{\lambda}$ is positive and of trace-class. 
	\hfill	$\diamondsuit$
\end{example}

One of the aims here is to study some spectral properties of Helson matrices induced by a measure $\mu.$ The scattering theory offers one possible avenue for doing so. For that purpose, the following two propositions collect a few results from scattering theory (e.g., \cite{Ka}, \cite[Propositions~1-3]{RS}). 

\begin{proposition}\label{wave-chain-complete}
	Let $H_1$ and $H_2$ be two $($possibly unbounded$)$ self-adjoint operators on a Hilbert space $\mathscr H$ and let $P_{ac}(H_j)$ be the projections onto the absolutely continuous part of the Hilbert space $\mathscr H$ with respect to $H_j, j =1,2.$
	\begin{itemize}
		\item [(i)] Set the wave operators,
		\beqn
		\Omega_{\pm}(H_2, H_1) = \text{s-}\lim_{t \rar {\pm \infty}} e^{itH_{2}} e^{-itH_1}P_{ac}(H_1),
		\eeqn
		if they exist. Then the operators $\Omega_{\pm}(H_2, H_1)$ are partial isometries with initial space $P_{ac}(H_1)\mathscr H$ and final ranges $R(\Omega_{\pm}(H_2, H_1)) \subseteq P_{ac}(H_2)\mathscr H.$ Furthermore, $\Omega_{\pm}(H_2, H_1)$ maps the domain $\mathcal D(H_1)$ of $H_1$ onto the domain $\mathcal D(H_2)$ of $H_2$ and 
		\beqn
		H_2\Omega_{\pm}(H_2, H_1)  = \Omega_{\pm}(H_2, H_1) H_1 \quad \text{on $\mathcal D(H_1)$}.
		\eeqn
		\item [(ii)] $($Chain rule$)$ Let $H_j~ (j=1,2,3)$ be three $($possibly unbounded$)$ self-adjoint operators and assume furthermore that the wave operators  $\Omega_{\pm}(H_3, H_2)$ and $\Omega_{\pm}(H_2, H_1)$
	exist. Then, $\Omega_{\pm}(H_3, H_1)$ exist and 
	\beqn
	\Omega_{\pm}(H_3, H_1) = \Omega_{\pm}(H_3, H_2)\Omega_{\pm}(H_2, H_1).
	\eeqn	
	\item [(iii)] $($Completeness$)$ The wave operators $\Omega_{\pm}(H_2, H_1)$ are said to be complete if
	\beqn
	\text{Range}(\Omega_{+}(H_2, H_1)) = \text{Range}(\Omega_{-}(H_2, H_1)) = P_{ac}(H_2)\mathscr H,
	\eeqn
	and in such a case, $H_{2, ac}$ is unitarily equivalent to $H_{1, ac}.$
	These wave operators are complete if and only if $\Omega_{\pm}(H_1, H_2)$ exist and in such a case, $\Omega_{\pm}(H_1, H_2) = \Omega_{\pm}(H_2, H_1)^*.$
	\end{itemize}
\end{proposition}

The next set of propositions gives the trace-class criterion for ensuring the existence and completeness of a pair of self-adjoint operators due to Kato-Birman-Rosenblum (see \cite{Bir}, \cite{Ka} and \cite[Theorem~XI.8]{RS}). 

\begin{proposition}\label{set-of-prop}
	 $\mathrm{(i)}$ Let $H_j~ (j=1,2)$ be two $($possibly unbounded$)$ self-adjoint operators such that $H_2 -H_1$ is of trace class. Then, $\Omega_{\pm}(H_2, H_1)$ exist and are complete. Furthermore, $H_{2, ac}$ is unitarily equivalent to $H_{1, ac},$ and $\sigma_{ac}(H_1) = \sigma_{ac}(H_2).$
	 
	 $\mathrm{(ii)}$ Let $H_j~ (j=1,2,3)$ be three $($possibly unbounded$)$ self-adjoint operators such that $H_3-H_2$ and $H_2-H_1$ both are of trace-class. Then, all the three wave operators $\Omega_{\pm}(H_3, H_2), \Omega_{\pm}(H_2, H_1)$ and $\Omega_{\pm}(H_3, H_1)$ exist and are complete, and furthermore,
	 \beqn
	 \Omega_{\pm}(H_3, H_1) = \Omega_{\pm}(H_3, H_2)\Omega_{\pm}(H_2, H_1),
	 \eeqn
	 or, equivalently, $\Omega_{\pm}(H_3, H_2) = \Omega_{\pm}(H_3, H_1) \Omega_{\pm}(H_2, H_1)^*.$ Also, $H_{j, ac}~ (j=1,2,3)$ are mutually unitarily equivalent.
\end{proposition}

\begin{remark}
	Clearly, Proposition~\ref{set-of-prop}(ii) is a consequence of Proposition~\ref{wave-chain-complete} and Proposition~\ref{set-of-prop}(i). It is obvious that if $H_3-H_2$ and $H_2-H_1$ both are of trace class, then $H_3-H_1$ is also of trace-class and by Proposition~\ref{set-of-prop}(i), $\Omega_{\pm}(H_3, H_1)$ will exist and be complete. It's importance in our context lies in the facts that for Helson matrices, the condition for verifying the trace-class property would need the positivity of the difference operators (see Proposition~A(i) of the Appendix). However, in some situations, this issue can also be addressed directly (see Theorem~\ref{trace-prop-gen}).
\end{remark} 

The next theorem considers a class of Helson matrices where the above two propositions can be applied. 

\begin{theorem}\label{scattering-weight-leb}
$\mathrm{(i)}$ Let $\mu_j~ (j=1,2)$ be two Borel measures on $(0, \infty)$ with $\lambda$ denoting the Lebesgue measure. Furthermore, let $\mu_1$ and $\mu_2$ be absolutely continuous with respect to $\lambda$ and with densities $\eta_j$ satisfying the condition in Theorem~\ref{weight-leb-bdd-comp}$(ii),$ and $\eta_j(0_+) > 0$ for $j =1,2.$ If we set $H_j$ as the corresponding Helson matrix $H_{\mu_j}~ (j=1,2)$ and set $H_0 = H_{\lambda},$ then $H_1-\eta_1(0_+)H_0$  and $H_2-\eta_2(0_+)H_0$ are of trace-class.

$\mathrm{(ii)}$ Both the wave operators $\Omega_{\pm}(H_1, \eta_1(0_+)H_0)$ and $\Omega_{\pm}(H_2, \eta_2(0_+)H_0)$ exist and are complete. If furthermore $\eta_1(0_+) = \eta_2(0_+) = \gamma>0,$ then $\Omega_{\pm}(H_2, H_1)$ exist, are complete and 
\beqn
\Omega_{\pm}(H_2, H_1) = \Omega_{\pm}(H_2, \gamma H_0) \Omega_{\pm}(H_1, \gamma H_0)^*.
\eeqn
Moreover, $\sigma_{ac}(H_j) = [0, \gamma \pi],~ j =1,2.$
\end{theorem}
\begin{proof}
For $j =1,2,$ since 
\beqn
	\Big|\eta_j\Big(\frac{y}{x}\Big) - \eta_j(0_+)\Big| \Le D(p) \Big(\frac{y}{x}\Big)^pg(y),\quad y > 0,~x > K >0
\eeqn
with $\int_0^\infty y^pg(y)e^{-y}dy < \infty,$ part (i) follows from Theorem~\ref{weight-leb-bdd-comp}(ii). The first part of (ii) follows from Proposition~\ref{set-of-prop}(i), and since $\eta_1(0_+) = \eta_2(0_+),$ the second part follows from Proposition~\ref{set-of-prop}(ii). Finally, it also follows from Proposition~\ref{set-of-prop}(ii) that $H_{j,ac}~ (j=1,2)$ are unitarily equivalent to $\gamma H_0$ implying that 
\beqn
\sigma_{ac}(H_j) = \sigma(\gamma H_0) =  \gamma \sigma_{ac}(H_0)= [0, \gamma \pi], \quad j =1,2,
\eeqn
as by \cite[Theorem~1.1]{PP-1}, $H_0$ is spectrally absolutely continuous.
\end{proof}

\begin{remark}
Since any densities $\eta_j~(j=1,2)$ given in  Example~\ref{exam-weight-leb}(iii)-(v) satisfy $\eta_j(0_+)=1,$ as a consequence of Theorem~\ref{scattering-weight-leb}, we get that $H_{\mu_j, ac}$ is unitarily equivalent to $H_{\lambda}.$ However, in Example~\ref{exam-weight-leb}(i)\&(ii), if we consider $\eta(t) = 1+t^p$ or $1+\ln(1+t),$ then the corresponding $H_{\mu, ac}$ will be unitarily equivalent to $H_{\lambda}.$
\end{remark}

Even when we do not have $\mu_1-\mu_2 \Ge 0$ or $\mu_2-\mu_1 \Ge 0,$ we see that ${H}_{\lambda}$ plays a central role in identifying the absolutely continuous spectra for certain choices of $\mu_1$ and $\mu_2.$
\begin{example}
	For $j =1,2,$ let $\mu_j(dt) = (1+t^{p_j})dt$ be such that $p_j>0.$ Note that neither $\mu_1-\mu_2 \Ge 0$ nor $\mu_2-\mu_1 \Ge 0.$ 
	Since $\mu_j(dt) - \lambda = t^{p_j}dt,$ by Example~\ref{exam-weight-leb}(i),  ${H}_{\mu_j} - {H}_{\lambda}$ is of trace-class.
	By Proposition~\ref{set-of-prop}(i), 
	\beqn
	\Omega_{\pm}(H_{\mu_j}, H_\lambda) = \text{s-}\lim_{t \rar {\pm \infty}} e^{iH_{\mu_j}t} e^{-iH_\lambda t}, \quad j = 1,2
	\eeqn
	exist and are complete. Therefore, $H_{\mu_j, ac}$ is unitarily equivalent to $H_{\lambda}.$ Furthermore, by Proposition~\ref{wave-chain-complete}(ii), we get that  \beqn
	\Omega_{\pm}(H_{\mu_1}, H_{\mu_2}) = \Omega_{\pm}(H_{\mu_1}, H_\lambda) \Omega_{\pm}(H_{\mu_2}, H_\lambda)^*.
	\eeqn
	This happens though  the measures $\mu_1(dt ) = (1+ t^p_1) dt$ and $\mu_2(dt) = (1+ t^p_2) dt,$ which have no domination between each other. In fact, even if $p_1 < p_2 ,~ t^p_1$    will compare differently  with $t^p_2$  depending on if $t<1$ or $t > 1.$
	\hfill $\diamondsuit$
\end{example}

\begin{example}
	Let $\mu_1$ and $\mu_2$ be two positive measures on $(0, \infty)$ such that $\mu_1 - \mu_2$ is a positive measure. Suppose that there exists $\epsilon > 0$ such that $\mu_1|_{[0, \epsilon)} = \mu_2|_{[0, \epsilon)}.$ Then, for any integer $m \Ge 2,$ 
	\beqn
	\widehat{\mu_1 - \mu_2}(2 \ln m) = \int_\epsilon^\infty m^{-2t} d(\mu_1 - \mu_2)(t)
	 &\Le& m^{-\epsilon} \int_\epsilon^\infty m^{-t}d(\mu_1 - \mu_2)(t)\\
	&\Le& m^{-\epsilon} \int_\epsilon^\infty 2^{-t}d(\mu_1 - \mu_2)(t),
	\eeqn
	which implies that \eqref{trace-cond} is satisfied, and hence by Proposition~A(i) in the Appendix, ${H}_{\mu_1} - {H}_{\mu_2}$ is a trace-class operator. Therefore, by Proposition~\ref{set-of-prop}(i), $\Omega_{\pm}(H_{\mu_1}, H_{\mu_2})$ exist, are complete and  $\sigma_{ac}(H_{\mu_1}) = \sigma_{ac}(H_{\mu_2}).$
    As a special case, if $\mu(dt) = \lambda|_{[0, 1)} + e^{-at}dt|_{[1, \infty)},$ then $\lambda - \mu$ is positive and
$\lambda|_{[0,1)} = \mu|_{[0,1)},$
which implies that $H_{\lambda} - H_\mu$ is of trace class. Therefore, $\sigma_{ac}({H}_{\mu}) = [0, \pi]$ as  $\sigma_{ac}({H}_{\lambda}) = [0, \pi].$    
		\hfill $\diamondsuit$
\end{example}

\section*{Appendix: Some known facts about Helson matrices}

In this appendix, we provide proofs of several established properties of Helson matrices that are difficult to locate in the existing literature.\\


\noindent \textbf{Proposition~A.} Let $H_{\alpha}:=(\alpha(mn))_{m,n=2}^\infty$ be a bounded operator on $\ell^2(\mathbb N_2).$ 
	Then, the following statements are valid$:$
	\begin{enumerate}
		\item[$\mathrm{(i)}$] if $H_{\alpha}$ is positive on $\ell^2(\mathbb N_2),$ then $H_{\alpha}$ is of trace-class if and only if 
		\beq \label{trace-cond}
		\sum_{m=2}^\infty \alpha(m^2)  < \infty,
		\eeq
		\item[$\mathrm{(ii)}$] 
		$H_{\alpha}$ is of Hilbert-Schmidt class if and only if
		\beq \label{Hilb-schm-cond}
		\sum_{m=4}^\infty |\alpha(m)|^2(d(m) -2) < \infty,
		\eeq
		where $d(\cdot)$ denotes the standard divisor function,
		\item[$\mathrm{(iii)}$] if $H_{\alpha}$ belongs to the Schatten $p$-class for some $2 < p < \infty,$ then 
		\beq \label{Schatten-p-condition}
		\sum_{n=2}^{\infty} \left(\sum_{m=2}^{\infty}|\alpha(mn)|^2\right)^{\frac{p}{2}} < \infty,
		\eeq
		\item[$\mathrm{(iv)}$] if $H_{\alpha}$ belongs to the Schatten $p$-class for some $p \in (1, \infty) \backslash \{2\},$ then 
		\beq \label{Schatten-p-condition-p}
		\sum_{m=2}^{\infty} |\alpha(m^2)|^{p} < \infty,
		\eeq
		\item[$\mathrm{(v)}$] for $p \in (0,2)\backslash \{1\},$ if \eqref{Schatten-p-condition} holds or 
		\beq \label{Schatten-p-condition-sa}
		\sum_{m,n=2}^\infty |\alpha(mn)|^p < \infty,
		\eeq
		then $H_{\alpha}$ belongs to the Schatten $p$-class.
	\end{enumerate}
	
\begin{proof} 
(i) 
Since $H_\alpha$ is positive and 
\beq
\label{H-ns-action}
\inp{H_\alpha e_m}{e_n} = \alpha(mn), \quad m, n \Ge 2,
\eeq
part (i) follows from  \cite[Theorem~4.1]{HKZ}.
	
	(ii) For any integer $m \Ge 2,$ 
	\beq 
	\|H_\alpha e_m\|_2^2 = \sum_{n=2}^\infty |\alpha(mn)|^2.
	\label{H-ns-action-square}
	\eeq
Since the series $\sum_{m,n=2}^\infty |\alpha(mn)|^2$ is indexed by $m,n \in \mathbb N_2$ and  $d(k)$ counts divisors including $1$ and $k$ for any integer $k \Ge 1$, it follows that $$\sum_{m,n=2}^\infty |\alpha(mn)|^2 = \sum_{m=4}^\infty |\alpha(m)|^2r(m),$$
	where $r(m) = d(m) - 2$ for all $m \Ge 4$. Thus, by \cite[Theorem~4.2]{HKZ}, $H_\alpha$ is Hilbert-Schmidt if and only if \eqref{Hilb-schm-cond} holds.
	
	(iii) Since any operator in the Schatten $p$-class is compact, by \eqref{H-ns-action-square} and \cite[Theorem~A]{HKZ}, we get that $\{\|H_\alpha e_m\|\}_{m=2}^\infty \in \ell^p(\mathbb N_2),$ proving \eqref{Schatten-p-condition}.
	
	(iv) By \eqref{H-ns-action},
	\beqn
	\inp{H_\alpha e_m}{e_m} = \alpha(m^2), \quad m\Ge 2,
	\eeqn
	which together with the assumption and \cite[Theorem~2.7]{HKZ} yields \eqref{Schatten-p-condition-p}.
	
	(v) Note that the first assumption implies that $\{\|H_\alpha e_m\|\}_{m=2}^\infty \in \ell^p(\mathbb N_2),$ which together with \cite[Theorem~B]{HKZ} gives the Schatten $p$-class of $H_\alpha.$  
	For the remaining part, note that \eqref{Schatten-p-condition-sa} together with \eqref{H-ns-action} gives that 
	\beqn
	\sum_{m,n=2}^\infty \inp{H_\alpha e_m}{e_n}^p < \infty.
	\eeqn
	Therefore, since $H_\alpha$ is self-adjoint, by \cite[Theorem~3.3]{HKZ}, $H_\alpha$ is in the Schatten $p$-class.
\end{proof}

%

In order to understand $H_\mu$ for positive measure $\mu$, we in the spirit of the proof of \cite[Theorem~2.2]{PP-1}, introduce the linear operator $\mathcal N_\mu$  defined as:
\beqn
\mathcal D(\mathcal N_\mu) := \Big\{x=\{x_m\}_{m \Ge 2} \in \ell^2(\mathbb N_2) : t \mapsto \sum_{m=2}^\infty x_m m^{-1/2-t} \in L^2((0, \infty), \mu) \Big\}
\eeqn 
and for $x=\{x_m\}_{m \Ge 2} \in \mathcal D(\mathcal N_\mu),$ 
\beq \label{N-mu-mod}
(\mathcal N_\mu x)(t) = \sum_{m=2}^\infty x_m m^{-1/2-t}, \quad t > 0.
\eeq
 
\noindent \textbf{Lemma~B.}
Let $\mu$ be a positive Borel measure on $(0,\infty)$. The
$\mathcal N_\mu$, defined by \eqref{N-mu-mod}, is  a closed densely defined operator on $\ell^2(\mathbb N_2)$. The Helson  matrix $H_\mu$ defined in Definition~\ref{H-mu} is therefore a densely defined positive operator in $\ell^2(\mathbb N_2)$.

\begin{proof}
	Note that $\mathcal D(\mathcal N_\mu)$ contains $\text{span}\{e_n : n \Ge 2\},$ which implies that $\mathcal N_\mu$ is a densely defined linear operator. 
	We next show that $\mathcal N_\mu$ is closed. To see this, let $\{f_n\}_{n = 1}^\infty \subseteq \mathcal D(\mathcal N_\mu)$ be a sequence such that $f_n \rightarrow f$ in $\ell^2(\mathbb N_2)$ and $\mathcal N_\mu f_n \rightarrow g$ in $L^2((0, \infty), \mu)$ as $n \rar \infty.$ Then, by the Cauchy-Schwarz inequality,
	\beqn
	\Big|\mathcal N_\mu f_n(t) - \sum_{m=2}^\infty f(m) m^{-1/2-t}\Big| \Le \|f_n - f\|_2 \sqrt{\zeta(1 + 2t)}, \quad t > 0,
	\eeqn
	where $\zeta$ denotes the Riemann zeta function. 
	This yields that $\mathcal N_\mu f_n \rar \sum_{m=2}^\infty f(m) m^{-1/2-t}$ as $n \rar \infty$ pointwise on $(0, \infty).$ Since a subsequence of $\{\mathcal N_\mu f_n\}_{n = 1}^\infty$ converges to $g$ pointwise $\mu$-a.e., we get
	\beqn
	g(t) =  \sum_{m=2}^\infty f(m) m^{-1/2-t}, ~\mu\text{-a.e.},
	\eeqn
	and hence, $f \in \mathcal D(\mathcal N_\mu)$ and $\mathcal N_\mu f =g.$ Thus, $\mathcal N_\mu$ is a densely defined closed linear operator with domain $\mathcal D(\mathcal N_\mu).$ 
	By von Neumann's theorem (see \cite[Theorem~V.3.24]{Ka}), \beq\label{N-mu}
	\text{$\mathcal N_\mu^*\mathcal N_\mu$ is a densely defined positive and self-adjoint operator.}
	\eeq
	We next show that $H_\mu = \mathcal N_\mu^*\mathcal N_\mu$ on $\mathcal D(\mathcal N_\mu^*\mathcal N_\mu)$ as the above matricial expression. Indeed, for every $x \in \mathcal D(\mathcal N_\mu^*\mathcal N_\mu) \subseteq \mathcal D(\mathcal N_\mu),$
	by Fubini's theorem,
	\beqn
	\inp{H_\mu(x)}{e_n} = 
	\sum_{m=2}^\infty x_m H_\mu(mn)
	&=& \int_0^\infty \mathcal N_\mu(x)(t) n^{-1/2-t} \mu(dt)\\
	&=&  
	\inp{\mathcal N_\mu(x)}{\mathcal N_\mu(e_n)}\\
	&=& \inp{\mathcal N_\mu^*\mathcal N_\mu(x)}{e_n}, \quad n \Ge 2,
	\eeqn
	which, by \eqref{N-mu}, completes the proof. 
\end{proof}

The following facts have been recorded in the discussion on \cite[p.~164]{PP-1}. Since we could not locate their proofs in the literature, we include them here for the sake of completeness.\\

\noindent \textbf{Proposition~C.} Let $\mu$ be a regular positive Borel measure $\mu$ on $(0, \infty)$ such that
$\int_{0}^\infty 2^{-t} \mu(dt) < \infty.$ Then, 
\begin{itemize}
	\item[$\mathrm{(i)}$] if $(\ln n) \widehat{\mu}(\ln n) \rar \infty$ as $n \rar \infty,$ then $H_\mu$ is a positive unbounded self-adjoint operator,
	\item[$\mathrm{(ii)}$]  if $(\ln n) \widehat{\mu}(\ln n) \rar 0$ as $n \rar \infty,$ then $H_\mu$ is compact.
\end{itemize}

\begin{proof}
(i): We know by Lemma~B that $H_\mu$ is a positive densely defined self-adjoint operator, so we need to prove only unboundedness. By the hypothesis, for each $C > 0,$ there exists a positive integer $N = N(C)$ such that
\beq \label{evetually-lower}
\widehat{\mu}(\ln (mn)) > C (\ln(mn))^{-1}~\text{whenever}~ mn \Ge N.
\eeq
Now set for such $N,$ a vector $x^{(N)} \in \ell^2(\mathbb N_2)$ by 
$
x_n^{(N)} = (n \ln n)^{-\frac{1}{2}},~ N \Le n \Le N^4.
$
Clearly, $\|x^{(N)}\|_2^2$ diverges to $\infty$ as $N \rar \infty.$ Next, for every integer $m \Ge 2,$ 
\beqn
(H_\mu x^{(N)})_m = \sum_{n=N}^{N^4} \frac{\widehat{\mu}(\ln (mn)) x_n^{(N)}}{\sqrt{mn}} \overset{\eqref{evetually-lower}}> 
\frac{C}{\sqrt{m}} \sum_{n=N}^{N^4} \frac{1}{n \sqrt{\ln n} \ln(mn)}.
\eeqn
Thus, by using \eqref{int-test-est} and letting $t = \sqrt{\ln x},$ 
\beqn
\inp{x^{(N)}}{H_\mu x^{(N)}} &>& C\sum_{m=N}^{N^4} \frac{1}{m \sqrt{\ln m}} \sum_{n=N}^{N^4} \frac{1}{n \sqrt{\ln n} \ln (mn)}\\ 
&\Ge& 
C\sum_{m=N}^{N^4} \frac{1}{m \sqrt{\ln m}} \int_{N}^{N^4} \frac{1}{x \sqrt{\ln x} \ln (mx)}dx\\
&=&  2C\sum_{m=N}^{N^4} \frac{1}{m \sqrt{\ln m}} \int_{\sqrt{\ln N}}^{2\sqrt{\ln N}}\frac{1}{t^2+\ln m}dy\\ &=&  2C\sum_{m=N}^{N^4} \frac{1}{m\ln m} \Big(\tan^{-1}\big(2\sqrt{\frac{\ln N}{\ln m}}\big) - \tan^{-1}\big(\sqrt{\frac{\ln N}{\ln m}}\big) \Big)\\
&=& 2C\sum_{m=N}^{N^4} \frac{1}{m\ln m} \tan^{-1}\Big(\frac{\sqrt{\frac{\ln N}{\ln m}}}{1+2\frac{\ln N}{\ln m}}\Big).
\eeqn
Recall that $\tan^{-1}(\cdot)$ is an increasing function. Next, since $\ln N \Le \ln m \Le 4\ln N,$
$
\sqrt{\frac{\ln N}{\ln m}} \Ge 2^{-1} ~~\text{and}~~ \frac{\ln N}{\ln m} \Le 1,
$
and hence 
\beqn
\inp{x^{(N)}}{H_\mu x^{(N)}}  \Ge 
2C \tan^{-1}\Big(\frac{1}{6}\Big) \sum_{m=N}^{N^4} \frac{1}{m\ln m} = 2C \tan^{-1}\Big(\frac{1}{6}\Big)\|x^{(N)}\|_2^2.
\eeqn
Since $C>0$ is arbitrary, $H_\mu$ is not bounded.

(ii): Define an infinite matrix for every $N \in \mathbb N_2,~ H_\mu^{(N)}$ in $\ell^2(\mathbb N_2)$ as
\beqn
H_{\mu,m,n}^{(N)} = 
\begin{cases}
	H_{\mu,m,n}, & 2 \Le m,n\Le N,\\
	0,& \text{otherwise} 
\end{cases}
\eeqn
so that $H_\mu^{(N)}$ is $N$-truncation of $H_\mu.$ Then, 
\beqn
(H_\mu - H_\mu^{(N)})_{m,n} = 
\begin{cases}
	0, & 2 \Le m,n\Le N,\\
	H_{\mu,m,n}, & \text{if $m$ or $n \Ge N+1$}. 
\end{cases}
\eeqn
Since by hypothesis $(\ln n)\widehat{\mu}(\ln n) \rar 0$ as $n \rar \infty,$ it follows that $\widehat{\mu}(\ln n) \Le D (\ln n)^{-1}$ for some $D > 0$ and all integers $n \Ge 2,$ and hence, $H_\mu$ (and of course $H_{\mu}^{(N)}$ as well) are bounded, by Theorem~\ref{l^2_suff_con}(i).

As in part (i) of Theorem~\ref{l^2_suff_con}, we apply Proposition~\ref{estimate_norm} on $H_\mu - H_\mu^{(N)},$ using the same auxiliary function $t_m = \frac{1}{\sqrt{m \ln m}}~(m \Ge 2)$ and note that the hypothesis implies that for arbitrary $\epsilon > 0,$ there exists $N_0 \in \mathbb N$ such that $\ln(mn)\widehat{\mu}(\ln mn) < \epsilon$ whenever $mn > N_0.$ Thus, 
\beqn
\sup_{n \Ge 2} \frac{1}{t_n}\sum_{\substack{m=2 \\ mn > N_0}}^\infty \frac{\widehat{\mu}(\ln(mn))}{\sqrt{mn}}t_m
&=& \sup_{n \Ge 2} \sqrt{\ln n}\sum_{\substack{m=2 \\ mn > N_0}}^\infty \frac{\widehat{\mu}(\ln(mn))}{m\sqrt{\ln m}}\\ 
&\Le& \epsilon \sup_{n \Ge 2} \sqrt{\ln n}\sum_{m=2}^\infty \frac{1}{m\sqrt{\ln m} \ln (mn)},
\eeqn
since $(H_\mu - H_\mu^{(N_0)})_{m,n} = 0$ whenever $mn \Le N_0.$ Therefore, by Proposition~\ref{estimate_norm} and \eqref{int-test-est},
\beqn
\|H_\mu - H_\mu^{(N)}\| \Le \|H_\mu -H_\mu^{(N_0)}\| \Le \pi \epsilon
\eeqn
for all $N> N_0,$ which gives that $\|H_\mu - H_\mu^{(N)}\| \rar 0$ as $n \rar \infty.$ On the other hand, it is clear that $H_\mu^{(N)}$ is Hilbert-Schmidt for every $N \in \mathbb N_2$ and therefore $H_\mu$ is compact.
\end{proof}

\noindent 
\textit{Acknowledgment:} The first author (SC) had the privilege of engaging in numerous discussions with the late Professor Franciszek Hugon Szafraniec, starting from their first meeting at IWOTA in 2008. During multiple visits to the Jagiellonian University, it was a pleasure to interact with him and share subtly humorous conversations.
The second author (CKS) thanks both the Indian Statistical Institute, Bangalore 
and the Indian Institute of Technology, Bombay for having facilitated 
his visits to interact with the third author (KBS) to carry out this project. The third author acknowledges support from the Senior Scientist Scheme of the Indian National Science Academy. The authors thank the anonymous referee for several suggestions that improved both the content and the presentation.
\vskip.2cm
\noindent 
\textit{Statement and Declarations:}

\vskip.3cm

\noindent
{\bf Conflict of interest} The authors declare that they have no conflict of interest.

\vskip.3cm

\noindent 
{\bf Data Availability} No data was used for the research described in the article.

{}
\end{document}